\documentclass[psamsfonts]{amsart}
\usepackage[utf8]{inputenc}
\usepackage{amsfonts}
\usepackage{hyperref}
\hypersetup{
pdftitle={Boundary behavior},
pdfsubject={Mathematics, },
pdfauthor={Hector},
pdfkeywords={}
}
\usepackage{amsmath}
\usepackage{xcolor}
\usepackage{amsthm}
\usepackage{pdflscape}
\usepackage{pgfplots}
\usepackage{mathrsfs}

\newtheorem{thm}{Theorem}[section]
\newtheorem{cor}[thm]{Corollary}
\newtheorem{prop}[thm]{Proposition}
\newtheorem{lem}[thm]{Lemma}

\newtheorem{quest}[thm]{Question}
\newtheorem{claim}[thm]{Claim}

\theoremstyle{definition}
\newtheorem{defn}[thm]{Definition}

\theoremstyle{remark}
\newtheorem{rem}[thm]{Remark}

\makeatletter
\let\c@equation\c@thm
\makeatother
\numberwithin{equation}{section}

\date{June 2025}

\begin{document}

\author{ H\'ector N. Salas}

\title {Boundary behavior of analytic functions on certain Banach spaces}

\address{Department of Mathematical Sciences\\University of Puerto Rico, Mayag\"{u}ez, PR 00681-9018} \email{hector.salas@upr.edu}
\thanks{Many thanks to Professor Bulancea for conversations on the topics}
\date{}

\begin{abstract}
For Banach spaces of analytic functions on the unit disc in which the polynomials are dense and their point evaluations continuous, we prove the following: If they contain a function such that the limit superior of its modulus is infinity almost everywhere on the unit circle, then the same is true for a residual set of functions.
\end{abstract}
\maketitle

\section{Introduction}

The Banach spaces of analytic functions on the unit disc $\mathbb D$ that we considered satisfy the requirement that the polynomials are dense and the point evaluations (or some variant) are continuous. Our main result is that if those spaces contain a function such that the limit superior of its modulus is infinity almost everywhere on the unit circle $\mathbb T$, then the same is true for a residual set of functions. A key ingredient for the proof is  the Baire Category Theorem.

Examples of this class of spaces are the Hardy weighted spaces $S_{\nu}$ for $\nu<0,$ and the Dirichlet type spaces 
$\mathcal D_{p-1}^p$ for $2<p.$ 

Other authors have considered Baire Category type arguments to obtain properties of some classes of analytic or meromorphic functions. See, for instance, Bagemihl \cite{Ba} and Anderson \cite {An}. The aims and methods of their papers are different from ours.

Composition operators  have been studied by many mathematicians, see, for instance the books by Shapiro \cite {Sh} and by Cowen and MacCluer \cite{Co-Mac}.
Composition operators in several classes of Banach function spaces on $\mathbb D$ have been studied by several mathematicians, for instance Zorboska \cite{Zo}, Gallardo-Guti\'errez and Montes-Rodr\'iguez \cite{Ga-Mo}, Colonna and Mart\'inez-Avenda\~no \cite{Co-Ma}.

The rest of the paper is organized as follows.

In the second section, we recall the definition of three classes of spaces:

(i) The classical $H^p$ spaces.

(ii) The weighted Hardy spaces $S_{\nu}$ for $\nu\in \mathbb R.$ 

(iii) The Dirichlet type spaces $\mathcal D_{p-1}^p$ for $2<p.$ 

We introduce the concept of
 $L_1$-average continuous point evaluations, and prove two lemmas concerning point evaluations which will be used in the proof of the main theorems. 
 
 Bulancea and Salas showed in \cite{Bu-Sa} that if $\nu<0$ then  there is an $f\in S_{\nu}$ such that
\[\limsup |f(r_ne^{i\theta})|=\infty ~\text{for some sequence}~ r_n \to 1 ~\text{ everywhere on} ~ \mathbb T.\]
 Refining that construction, we can prove that the same is true for an $f \in \hat S_0=\cap_{\nu<0}S_{\nu}.$

In the third section, we prove our main results and obtain some of their consequences. 

One of them is that the classical Hardy space $H^2=S_0$ is a first category subset of $\hat S_0=\cap_{\nu<0}S_{\nu}.$ (This is a complete metric space when endowed with a metric coming from a denumerable family of norms.) 

Another consequence 
is that $H^p$ is a first category subset of $\mathcal D_{p-1}^p$ for $2<p.$ This is based 
in the following facts: 

i) The classical result of Littlewood and Paley \cite{Li-Pa} which says that $H^p \subset \mathcal D_{p-1}^p$ for $2<p.$ 

ii) Abkar \cite{Ab} showed that the polynomials are dense in $ D^p_{p-1}$.

iii) Girela and Pela\'ez \cite{Gi-Pe} showed that if $2<p,$ then there exists a function  $f \in D^p_{p-1}$ with
 $\lim_{r \uparrow 1} |f(re^{i\theta})|=\infty$
a.e. on $\mathbb T.$  

The last section consists of a few questions and comments.

\section{ Preliminary results}
Recall that residual sets contain dense $G_{\delta}$ sets. A $G_{\delta}$ set is a denumerable intersection of open sets.

For the unit circle $\mathbb T$ the Lebesgue normalized measure is denoted by $dm(\theta)=\frac{1}{2\pi}d \theta$. If $E\subset \mathbb T$ is Lebesgue measurable, then $|E|$ denotes its measure. 

The classes of analytic functions on $\mathbb D$
such that for $0<p<\infty$ 
\[||f||_p=\lim_{r\to 1} \Big(\int|f(re^{i\theta})|dm(\theta)\Big)^{\frac{1}{p} }<\infty\]
are called Hardy spaces $H^p,$ a standard reference is the book by Duren \cite{Du}. They are separable Banach spaces for $1\leq p$
with norm given by the above formula. The set of bounded analytic functions on $\mathbb D$
is also a Banach space denoted by $H^{\infty}.$ It is not separable. The norm of $f$ is given by 
\[||f||_{\infty}=\sup_{r \to 1}\max \{|f(re^{i\theta})|:e^{i\theta} \in \mathbb T \}.\]

A theorem of Fatou says that the radial limit 
(actually nontangential) 
\[\lim_{r \to 1}f(re^{i\theta})=f(e^{i\theta})\]
exists almost everywhere for $f \in H^{\infty}.$
Therefore the radial limit exists also for the class of quotients of bounded functions, the Nevannlina class $\mathcal {N}$. Moreover,
$H^p\subset \mathcal{N}$ and also
\[
||f||_p=\Big(\int|f(e^{i\theta})|^{p}dm(\theta)\Big)^{1/p}.
\]
Thus $H^p$ can be identified with $\{f \in L^p(\mathbb T,dm(\theta)):\hat f(n)=0 ~\text{ for } n<0\}.$
In particular, the classical Hardy space $H^2$ is a Hilbert space.

For each sequence of positive numbers $\beta=\{\beta_n\}_n$ the weighted Hardy space $H^2(\beta)$ is the Hilbert space of functions analytic on $\mathbb{D}$ for which the norm induced by the inner product
    $$\big{<}\sum_{n=0}^\infty{a_nz^n}, \sum_{n=0}^\infty{b_nz^n}\big{>}=\sum_{n=0}^\infty{a_n\overline{b_n}\beta_n^2}$$
    is finite. Thus for $f(z)=\sum_0^{\infty}a_nz^n$ the norm is given by 
    \[\Big(\sum_0^{\infty}|a_n|^2\beta^2\Big)^{1/2}=||f||_{H^2(\beta)}.\] The monomials form a complete orthogonal system and so they are dense in $H^2(\beta).$ Also, convergence in $H^2(\beta)$ implies uniform convergence on compact subsets of the unit disk. We will focus on  $S_\nu$, the weighted Hardy spaces with weights $\beta_n=(n+1)^\nu$, where $\nu$ is a real number. If $\nu_1>\nu_2$, then $S_{\nu_1}$ is strictly contained in $S_{\nu_2}$; if $\nu>\frac{1}{2}$, then $S_\nu$ is contained in the disk algebra $\mathcal{A}$.
For $\nu=0$ we recover the classical Hardy space; that is, $H^2=S_0.$ We also see that $S_{1/2}$ is the classical Bergman space whereas $S_{-1/2}$ is the classical Dirichlet space. The Hardy weighted spaces are sometimes called weighted Dirichlet spaces. On p. 15 of \cite{Co-Mac}, these spaces are called Hardy weighted spaces if the norm is obtained from Taylor coefficients, Bergman spaces if the norm is obtained from $|f|$ and Dirichlet spaces if the norm is obtained from $|f^{\prime}|.$ The three norms are equivalent but not identical. 

The normalized area measure in $\mathbb D$ is denoted by $dA(z)=\frac{1}{\pi} dxdy=\frac{1}{\pi}rdrd\theta $  where $z=x+iy=re^{i\theta}.$

For $-1<\alpha, ~0<p,$ the weighted Bergman space $\mathcal A^p_{\alpha}$  is the set of analytic functions on $\mathbb D$
contained in 
$L^P(\mathbb D,(1-|z|^2)^{\alpha}dA(z))$

When $1\leq p,\mathcal A^p_{\alpha}$  are Banach spaces with norm
\[||f||_{A^p_{\alpha}}=\Big((\alpha+1)\int _{\mathbb D}(1-|z|^2)^{\alpha}|f(z)|^p ~dA(z)\Big)^{1/p}.\]
$\mathcal A^2_0$ is the classical Bergman space.
A general reference for these spaces is the book by Hedenmalm, Korenblum and Zhu \cite{He-Ko-Zh}
 
 The space $\mathcal D^p_{\alpha}=\{f:f^{\prime}\in \mathcal A^p_{\alpha}\}$
 is said to be Dirichlet type if  $\alpha \leq p+1,$  \cite{Gi-Pe}. (They are also called weighted Dirichlet spaces \cite{Co-Ma}.)
 For $1\leq p$ they are separable Banach spaces when endowed with the norm 
\[||f||_{\mathcal D^p_{\alpha}}=|f(0)|+||f^{\prime}||_{A^p_{\alpha}}.\]
In particular, $\mathcal D^2_0$  is the classical Dirichlet space and $\mathcal D^2_1=H^2.$

As mentioned in the Introduction, we are interested in $\mathcal D^p_{p-1}$ when $2<p,$
in which case 
\[H^p\subset \mathcal D^p_{p-1}\]
according to \cite{Li-Pa} and the polynomials are dense according to \cite{Ab}.

The following lemma will be used in the proof of the first theorem.

\begin{lem} \label{First Lemma}
Let $E$ be a Banach space of continuous (real or complex) functions on a complete metric space $X$ in which point evaluations are continuous. If $K \subset X $  is compact, then there exists a constant $C_K$ for which
\[|f(x)|\leq C_K||f|| ~\text{ for all } x\in K.\]
\end{lem}

\begin{proof}
Let $L_x$ be
the bounded linear functional $L_x(f)=f(x).$  The Uniform Bounded Principle says that either
$\sup \{||L_x||:x \in K\} <\infty$ or $\sup\{L_x(f): x\in K\}=\infty$ for all $f\in E$
belonging to some dense $G_{\delta}$ in $E.$ But for each $f$ the $\sup\{|f(x)| :x \in K\}$ is finite because $f$ is continuous and $K$ is compact.
\end{proof}

\begin{defn} Let $E$ be a Banach space of analytic function on $\mathbb D.$
The point evaluations are $L^1$-average continuous if
\[\int_0^{2\pi} |f(re^{i\theta}|~dm(\theta )\leq C(r) ||f||\]
for all $f \in E$ and $C(r) \in \mathbb R^+$ for all $0\leq r_0 \leq r<1.$
\end{defn}

The following proposition will be used in the proof of the second theorem.

\begin{prop}
     The point evaluations are $L^1$-average continuous in $\mathcal D_{p-1}^p$ whenever $2<p.$
\end{prop}

\begin{proof} Since $f(z)=f(0)+\int_M f^{\prime}(\zeta)~d\zeta$ where $M$ is the segment from $0$ to $z$ we have that
  \begin{equation} \label{eq: first bound}  \int_0^{2\pi}|f(re^{i\theta})|~d\theta\leq
    2\pi|f(0)|+\int_0^{2\pi}\int_0^r|f^{\prime}(se^{i\theta})|~ds~d\theta\end{equation}
Let $q$ be the conjugate of $p;$ i.e, $\frac{1}{q}+\frac{1}{p}=1.$ Thus $1<q<2<p.$
We now use Holder's inequality for the integrand $s^{-1/p}|f^{\prime}  (se^{i\theta})|s^{1/p}$  
\begin{equation} \label{eq: Holder}
    \int_0^{2\pi}\int_0^r|f^{\prime}(se^{i\theta})|~ds~d\theta\leq L(r)\Big(\int_0^{2\pi}\int_0^r|f^{\prime}(se^{i\theta})|^ps~ds~d\theta\Big )^{1/p}\end{equation}
with $L(r)=\Big(\int_0^{2\pi}\int_0^rs^{-q/p}~ds~d\theta\Big )^{1/q}$
which is finite since $0<\frac{q}{p}<1.$
Since $1 \leq \frac{1-s^2}{1-r^2}$ for $0\leq s\leq r<1$ with $z=se^{i\theta}$ we have that $1\leq \Big(\frac{1-|z|^2}{1-r^2}\Big)^{p-1}.$ 

Below we use that $\frac{p-1}{p}=\frac{1}{q}$ and $\pi dA(z)=sdsd\theta.$ 
\[\Big(\int_0^{2\pi}\int_0^r|f^{\prime}(se^{i\theta})|^ps~ds~d\theta\Big )^{1/p}\leq\]

\[\Big(\frac{1}{1-r^2}\Big)^{1/q}\Big (\frac{\pi}{p} \Big )^{1/p}\Big(p\int_{\{|z|\leq r\}}(1-|z|^2)^{p-1}|f^{\prime}(z)|^pdA(z)\Big )^{1/p}\]

Thus by integrating in the whole disc $\mathbb D $ and inequality \eqref{eq: Holder}
\begin{equation}\label{eq: third bound}\int_0^{2\pi}\int_0^r|f^{\prime}(se^{i\theta})|~ds~d\theta\leq L(r)\Big(\frac{1}{1-r^2}\Big)^{1/q}\Big (\frac{\pi}{p} \Big )^{1/p}||f||_{\mathcal D^p_{p-1}}\end{equation}

Set 
\[C(r)=2\pi +L(r)\Big(\frac{1}{1-r^2}\Big)^{1/q}\Big (\frac{\pi}{p} \Big )^{1/p}.\]

Then, by using inequalities 
\eqref {eq: first bound} and \eqref {eq: third bound} we have that
\[\int_0^{2\pi}|f(re^{i\theta})|~d\theta \leq C(r) ~||f||_{\mathcal D^p_{p-1}}\]
which is what we wanted to prove.
\end{proof}

The following proposition shows that 
$H^2=S_0 $ is properly contained in $\hat S_0=\cap_{\nu<0}S_{\nu}.$ It will be used in Corollary 3.9. 

\begin{prop} There exists an analytic function $f$ on $\mathbb D$ which belongs to all $S_{-\nu}$ with $\nu>0$ and a sequence of positive radii $r_k\uparrow 1$ such that 
\[
\lim_{k\to \infty}\text{min}\{|f(z)|:|z|=r_k\}=\infty.
\]
\end{prop}

\begin{proof} We will construct a function $f$ such that $f\in S_{-\nu_k}$ where $\nu_k \downarrow 0$ when $k\to \infty.$

Let $f(z)=\sum_{k=1}^{\infty}c_kz^{n_k}$ satisfy that the sequence $\{n_k\}_k$ is going to 
$\infty$ very fast and also that the sequence of radii $\{r_k\}$ is increasing to 1 very fast. 

Let $\{c_k\}_k$ be a sequence of positive numbers such that
\begin{equation} \label{1}
    c_1>1 ~\text{ and } ~ c_k-\sum_{j=1}^{k-1}c_j>k.
    \end{equation} 

    The idea of the construction is that for 
    $|z|=r_k$ the dominant term in $|f(z)|$
is $c_kr_k^{n_k}.$

 In the k-step is first chosen $n_k$ and then is chosen $r_k.$ The process is as follows:
$n_1$ is chosen first and then $r_1$ such that 
\[\frac{c_1}{(n_1+1)^{\nu_1}}<\frac{1}{\sqrt{2}}
~\text {  and  }~c_1r_1^{n_1}>1.\]

Assume that $n_1,\cdots, n_k$ and $r_1,\cdots ,r_k$ have been  chosen such that 
\begin{equation} \label {2}
    \text{ for }~ 1\leq j\leq k,~~\frac{c_j}{(n_j+1)^{\nu_j}}<\frac{1}{\sqrt{2^j}}.
\end{equation}

\begin{equation} \label{3}
c_1r_1^{n_1}-\sum_{j=2}^kc_jr_1^{n_j}>1
\end{equation}

\begin{equation} \label{4}
     1<p<k \Longrightarrow c_pr_p^{n_p}-\sum_{1\leq j <p}c_j -\sum_{p<j\leq k} c_jr_p^{n_j}>p
     \end{equation}

\begin{equation} \label{5}
    c_kr_k^{n_k}-\sum_{j=1}^{k}c_j>k.
    \end{equation}

Now we find  $n_{k+1}$ large enough such that
\[\frac{c_{k+1}}{(n_{k+1}+1)^{\nu_{k+1}}}<\frac{1}{\sqrt{2^{k+1}}}~\text{ and }~c_1r_1^{n_1}-\sum_{j=2}^{k+1}c_jr_1^{n_j}>1\]
\[ 1<p<k+1 \Longrightarrow c_p-\sum_{1\leq j <p}c_j -\sum_{p<j\leq k+1} c_jr_p^{n_j}>p\]

We can now choose $r_{k+1}$ with $r_k<r_{k+1}<1$ but sufficiently near to 1 such that
\[c_{k+1}r_{k+1}^{n_{k+1}}-\sum_{j=1}^{k}c_j>k+1.
\] 
This shows that \eqref{1},  \eqref{2}, \eqref{3}, \eqref{4} and \eqref{5} are valid for all $k\in \mathbb N.$
If $|z|=r_p,$ then using \eqref{4} or \eqref{3}
if $p=1$ and allowing any $p$
\[p\leq c_p|r_p|^{n_p}-\sum_{1\leq j <p}c_j -\sum_{p<j} c_jr_p^{n_j}\leq |f(z)|\]
Thus
\[\min\{|f(z)|:|z|=r_p\}\geq p.\]

To conclude we need to show that if $\nu>0$ then
$f\in S_{-\nu}.$ 

Let $\nu>\nu_k>0.$ Therefore using inequality \eqref{2}
\[\sum_{j=1}^{\infty}\Big(\frac{c_j}{(n_j+1)^{\nu}}\Big)^2\leq \sum_{j=1}^{k-1}\Big(\frac{c_j}{(n_j+1)^{\nu}}\Big)^2+\sum_{j=k}^{\infty}\Big(\frac{c_j}{(n_j+1)^{\nu_j}}\Big)^2\leq\]
\[\sum_{j=1}^{k-1}\Big(\frac{c_j}{(n_j+1)^{\nu}}\Big)^2+\sum_{j=k}^{\infty}\frac{1}{2^j}<\infty.\]
This completes the proof of the proposition. \end{proof}

\section{Two flavors of the main theorem}

\begin{thm} Let $E$ be a Banach space of analytic functions on the unit disc $\mathbb D$ in which the  polynomials are dense and the point evaluations are continuous. Assume that there exists $f\in E$ such that   
\[|\{e^{i\theta}:\limsup_{r \to 1}|f(re^{i\theta}|=\infty\}|=1\]

Then

\[\{g\in 
E: |\{e^{i\theta}:\limsup_{r \to 1}|g(re^{i\theta}|=\infty\}|=1 \}\]
 is a residual set in $E.$
\end{thm}

\begin{proof}
    For each natural number $M$ we can find a radius 
    $r_M <1$ such that
\begin{equation} \label{measure of F_M}
F_M=\{e^{i\theta}:\exists r \leq r_M \text{ and  } |f(re^{i\theta})|\geq M\} \text { and  } ~|F_M|\geq 1-\frac{1}{M}.
\end{equation} 
Each $F_M$ is a closed set and therefore measurable. 

Let $P_n$ be a sequence of polynomials  dense in $E.$ For each $n,k \in \mathbb N$  let $B_{n,k}$ be a ball  centered at $P_n+\frac{1}{k}f$  with radius $\epsilon(n,k)$ to be determined momentarily.
For each $n,k$ we can find  $M=M(n,k)$
 so large that
 \begin{equation}\label {relation between M and k}
     \frac{M}{k}>k+||P_n||_{\infty}
+1 .\end{equation}
According to Lemma \ref{First Lemma} there exists $C_M$ such that for all $z$ with $|z|\leq r_M$ hold that
\begin{equation}\label{point evaluations}
    |h(z)|\leq C_M||h||  \text{   for all   } h\in E.
\end{equation}
We now choose $\epsilon(n,k)=min\{\frac{1}{k},\frac{1}{C_M}\}.$

\begin{claim}  If $z=re^{i\theta},~ r\leq r_M ,~|f(z)|\geq M$ and
 $g\in B_{n,k},$ then
 $|g(z)>k.$
 \end{claim}

 By inequality \eqref{point evaluations} and our choice of $\epsilon(n,k)$ we have that 
 \[| P_n(z)+\frac{1}{k}f(z)-g(z)|\leq C_M || P_n+\frac{1}{k}f-g||<C_M \epsilon(n,k)\leq 1 \]
 
 Using the fact that $|f(z)|>M,$ the above inequality, triangular inequality and \eqref{relation between M and k} 
 
 \[|g(z)\geq \frac{1}{k}|f(z)|-|P_n(z)|-| P_n(z)+\frac{1}{k}f(z)-g(z)|\geq\]
 \[\frac{M}{k}-||P_n||_{\infty}-1> k.\]
 The claim is now proved. 

 As a consequence we have that for any $g \in B(n,k)$

 \[
     F_M\subset A(n,k)=\{e^{i\theta}:\exists r\leq r_M \text{   and  } |g(re^{i\theta})|\geq k\}
 \]
and therefore using \eqref{measure of F_M} and \eqref{relation between M and k}
\begin{equation} \label{measure of A}
  |A(n,k)|\geq 1-\frac{1}{M}  \geq 1-\frac{1}{k^2}
\end{equation}

We now define the open sets
\[W_k=\cup_{1\leq n,k\leq j}B(n,j)\]
 which are dense since each polynomial $P_n$ is a limit point of $W_k.$ Using Baire's category theorem we have that
 \[W_{\infty}=\cap_{1\leq k}W_k\]
is a residual set in E.

\begin{claim} 
 If $g\in W_{\infty},$  then   $ |\{e^{i\theta}:\limsup_{r \to 1}|g({i\theta}|=\infty\}|=1.$
 \end{claim}

For each $k$ there are $n$ and $j\geq k$ such that  $g\in 
B(n,j).$ Thus we have $n_p,j_p$ where $j_p \to \infty$ and $g\in B(n_p,j_p)$ for all $p.$

Let $\cup _{p=s}A(n_p,j_p)=A_s.$ Then \eqref{measure of A} implies that $|A_s|=1.$ Let $\cap_{1\leq s}A_s=A$ and since $A_{s+1}\subset A_s$ it follows that $|A|=1.$ If $e^{i\theta}\in A_s,$ then there exists $ r$ such that $|g(re^{i\theta})|\geq j_s.$ Thus we have that
\[\limsup_{r \to 1}|g(re^{i\theta})|=\infty,\]
proving the second claim and therefore completing the proof of the theorem.
\end{proof}

\begin{cor}
 Let $\nu<0.$ The set 
 \[\{g\in S_{\nu}:\limsup_{r \to 1} |g(re^{i\theta})|=\infty ~\text{ for almost all }~ e^{i\theta}\}\] 
 is residual in $S_{\nu}$.
\end{cor}

\begin{proof}
In $S_{\nu}$ the polynomials are dense and the point evaluation continuous \cite{Co-Mac}. By Proposition 5.8 of \cite{Bu-Sa} there exists a function $f$ that satisfies the condition of the preceding theorem and therefore the conclusion of such a theorem is also obtained.    
\end{proof}

As we said in the introduction, $\hat S_0$ is a complete metric space. The requirement is that the norms $||~||_{S_{\nu_k}}$ satisfy $\lim_{k\to \infty} \nu_k=0$ and $\nu_k<0$ for all $k.$
\begin{cor} The set
$g \in \hat S_0=\cap_{\nu<0}S_{\nu}$ such that \[|\{e^{i\theta}:\limsup_{r\to 1}
|g(re^{i\theta} )| = \infty\}|=1\]  is
(i) residual in $\hat S_0$,
and therefore
(ii) $H^2 = S_0$ is a first category subset of $\hat S_0.$
\end{cor}

\begin{proof} (i) We used the function obtained in Proposition 2.7 and a similar approach that in the preceding theorem. In the definition of the radius of the ball $B(n,k)$ we use $||h||_{S_{\nu_k}}$ in \eqref{point evaluations}.

 (ii) The functions in $H^2$ are convergent almost everywhere in $\mathbb T,$ 
 but the functions that behave badly when approaching the boundary $\mathbb T$ is a residual set. This concludes the proof. 
\end{proof}

\begin{thm}
  Let $E$ be a Banach space of analytic functions on the unit disc $\mathbb D$ such that the polynomials are dense and the point evaluations are $L^1$-average continuous. Let $\phi$ be an increasing positive continuous function on $0\leq r_0 \leq r <1$ with 
  $\lim_{r \to 1}\phi(r)=\infty.$ 
  
  Assume further that there exist $f\in E$ and a sequence $r_p \to 1$ for which    \[\lim_{p\to \infty}\frac{1}{\phi(r_p)}min\{|f(r_pe^{i\theta})|:e^{i\theta} \in \mathbb T\}=\infty.\] 
 Then

\[\{g\in 
E: |\{e^{i\theta}:\limsup_{r \to 1}\frac{1}{\phi(r)}|g(re^{i\theta}|=\infty\}|=1 \}\]
 is a residual set in $E.$
\end{thm} 

\begin{proof}
We adapt the plan used in Theorem 3.1.

Let $\{P_n:n\in \mathbb N\}$ be a dense set of polynomials in $E.$
By hypothesis for each $r<1$ there exists $C(r)$  such that
\begin{equation}\label{L^1 average}
\int_0^{2\pi} |h(re^{i\theta}|~dm(\theta )\leq C(r) ||h||\end{equation}
for all $h \in E.$

For each $(n,k)$ choose $M=M(n,k)$ and $p(n,k)$
 such that
\begin{equation} \label{M is big}
 ||P_n||_{\infty}+k \leq M 
\end{equation}  

\begin{equation} \label{f(z) with |z|=r_p is big}
2kM^2<\frac{1}{\phi(r_{p(n,k)})}\min\{|f(r_{p(n,k)}e^{i\theta})|:e^{i\theta} \in \mathbb T\}.
\end{equation}
 
Let
$B(n,k)$ be the ball centered at $P_n +\frac{1}{k} f$ and radius 

\begin{equation}\label{radius of B_n is small}
\epsilon(n,k)=min\Bigl\{\frac{1}{k},\frac{1}{C(r_{p(n,k)})}\Bigr\}.
\end{equation} 

For $g\in B(n,k)$ let 
\[H=H(n,k)=\{e^{i\theta}:\frac{1}{\phi(r_{p(n,k)})}|g(r_{p(n,k)}e^{i\theta})|<M\}.\] 

Now we estimate $|H|,$ the Lebesgue measure of $H.$ Let $r=r_{p(n,k)}.$ Then the first inequality is due to
\eqref{f(z) with |z|=r_p is big}, and the second is the triangle inequality.  In the following we may assume that $\phi(r)>2$.

\[2M^2|H|\leq\int_H \frac{1}{k\phi(r)}|f(re^{i\theta})|dm(\theta)\leq\]

\[\frac{1}{\phi(r)}\Big (\int_H |g(re^{i\theta})|+ |P_n(re^{i\theta})+\frac{1}{k}f(re^{i\theta})-g(re^{i\theta})|+|P_n(re^{i\theta})|dm(\theta)\Big ).\]
We now get an upper bound for each of the last three summands:

The definition of $H$ implies that\[\frac{1}{\phi(r)}\int_H |g(re^{i\theta}dm(\theta)|\leq M|H|\leq M.\]

the fact that $g\in B_n,$ the radius \eqref{radius of B_n is small} and the $L^1$ average continuity 
\eqref{L^1 average} and also w.l.g we may assume that $r$ is sufficiently near to 1 such that $\phi(r)>2.$
\[\frac{1}{\phi(r)}\int_H  |P_n(re^{i\theta})+\frac{1}{k}f(re^{i\theta})-g(re^{i\theta})||dm(\theta)\leq\]
\[\frac{1}{\phi(r)}\int_{\mathbb T}  |P_n(re^{i\theta})+\frac{1}{k}f(re^{i\theta})-g(re^{i\theta})|dm(\theta)\leq \frac{1}{\phi(r)}\leq \frac{1}{2}M.\]
For the last summand we use \eqref{M is big}
\[\frac{1}{\phi(r)}\int_H |P_n(re^{i\theta})|dm(\theta)\leq \frac{1}{\phi(r)}\int_{\mathbb T} |P_n(re^{i\theta})|dm(\theta)\leq \frac{1}{2}M.\]

Consequently
\[2M^2|H|\leq 2M \Longrightarrow |H|\leq \frac{1}{M}  ~~~\text{ and }1-\frac{1}{M} \leq |\mathbb T\setminus H|.\]

As in the previous theorem the open sets
\[W_k=\cup_{1\leq n,k\leq j}B(n,j)\]
  are dense since each polynomial $P_n,$ is a limit point of $W_k.$ Using Baire's category theorem again we have that
 \[W_{\infty}=\cap_{1\leq k}W_k\]
is a residual set in E.

Let $g\in W_{\infty}.$ Arguing like in Claim 3.7
we obtain 
\[\limsup_{r \to 1}\frac{1}{\phi(r)}|g(re^{i\theta})|=\infty\]
for almost all $e^{i\theta} \in \mathbb T.$
\end{proof}

\begin{cor}
The space $H^p$ is a first category subset of $\mathcal D^p_{p-1}$ for $2<p.$  
 \end{cor}

\begin{proof} By Fatou's theorem a function in $H^p$
converges a.e on $\mathbb T.$ On the other hand, the density of the polynomials \cite{Ab}, Proposition 2.3,  and Theorem 3.5 in \cite{Gi-Pe} allow us to use the preceding theorem. 
\end{proof}

\section{ Concluding remarks}

\begin{quest}
    
Are the point evaluation continuous on $\mathcal D_{p-1}^p$ for $2<p?$ Can one modify the proof of Theorem 3.1 in 
\cite{Gi-Pe} to have continuity?
\end{quest}

\begin{quest}
     Does there exist a Banach space of analytic functions on $\mathbb D$ in which the point evaluations are not continuous but they are $L^1$-average continuous?
\end{quest}

\begin{quest}
     How does $S_{\nu}$ fit into $\hat S_{\nu}=\cap_{\mu<\nu}S_{\mu}$ in general? 
     When $\nu=0$ the answer is given by Corollary 3.9.
There we have a powerful tool in terms of convergence or not on the boundary.
\end{quest}
\begin{rem}
    
 It might be of interest if in Theorems 3.1 and 3.10 we could obtain limit (instead of limit superior ) in the  conclusions if there are  
limits in the assumptions.
\end{rem}

\begin{rem} How does $\cup_{\mu>\nu}S_{\mu}$ fit in $S_{\nu}?$ By Prop 5.3 \cite{Bu-Sa} each $S_{\mu}$ is compactly embedded in $S_{\nu}$ when {$\mu>\nu,$}
and is first category since $S_{\nu}$ is infinite dimensional. Since $\cup_{\mu>\nu}S_{\mu}=\cup_{\mu_n>\nu}S_{\mu}$
for any sequence $\lim_{n\to\infty}\mu_n=\nu,$
it follows that the union is a first category subset of $S_{\nu}.$
\end{rem}

\begin{rem} The density of the polynomials in both theorems could be replaced by the density of the disc algebra $\mathcal A.$
\end{rem}

\end{document}